\documentclass{amsart}
\usepackage{geometry, graphicx, layout, multicol, tikz, textcomp, wrapfig, float, xfrac, faktor, nicefrac, setspace, amsthm, epstopdf, xcolor, multirow, amssymb, amsmath, tensor, mathrsfs}
\usepackage[utf8]{inputenc}
\usepackage[english]{babel}
\usepackage[utf8]{inputenc}
\usepackage[english]{babel}
\RequirePackage{fix-cm}
\usepackage[fontsize=12pt]{fontsize}
\usepackage[pagewise]{lineno}
\theoremstyle{plain}
\usepackage{enumerate,bm}

\title{Groups with a Fixed Character Degree: The General Case}
\author{Mark L. Lewis \and Brandon Martin}
\address{Department of Mathematical Sciences, Kent State University, Kent, OH 44242}
\email{lewis@math.kent.edu}
\email{bmarti52@kent.edu}

\begin{document}
\onehalfspacing
\newtheorem{theorem}{Theorem}
\newtheorem{lemma}{Lemma}
\newtheorem*{ulemma}{Lemma}
\newtheorem{sublemma}{Sublemma}[lemma]
\newtheorem*{utheorem}{Theorem}
\newtheorem*{theoremone}{Theorem 1}
\newtheorem*{theoremtwo}{Theorem 2}

\maketitle

\begin{abstract} 
 We obtain arithmetic conditions which are satisfied by solvable groups admitting an abelian $\pi$-subgroup.  First, we extend a previous characterization by removing the square-free hypothesis on some fixed irreducible character degree.  We then show the same arithmetic conditions follow from a faithful action of an abelian $\pi$-subgroup on the $\pi'$-part of the Fitting subgroup, where $\pi$ is the set of prime divisors of the aforementioned character degree. 
\end{abstract}

\section{Introduction} 
We will let all groups be finite throughout this paper.  Let $G$ be a group of order $d(d+e)$ where $e>1$ is an integer and $d\in\text{cd}(G)$, where $\text{cd}(G)=\{\chi(1)\text{ $|$ }\chi\in\text{Irr}(G)\}$.  N. Hung, Lewis, and Schaeffer Fry \cite{LHS} proved that if $d$ is the degree of a complex irreducible character of $G$ where $|G|=d(d+e)$ for some integer $e>1$, then $|G|\leq e^4-e^3$.  Their proof uses the classification of finite simple groups.  Also, by earlier work by Isaacs \cite{Is}, the bound of $e^4-e^3$ is the best possible. \par 

The above results assume that $d\in\text{cd}(G)$.  We now consider this problem from the opposite perspective, as suggested by the authors of \cite{LHS}.  In particular, we will no longer assume that $d\in\text{cd}(G)$, and we look to characterize conditions on $G$ which result in $d$ being an irreducible character degree of $G$. \par 

Let $G$ be a group of order $d(d+e)$, where $d,e>1$.  Snyder \cite{Sny} classified all such groups with $e=2$ and $e=3$ which has $d$ as a complex irreducible character degree.
Durfee and Jensen \cite{DJ} classified all such groups with $4\leq e\leq 6$, as well as all irreducible character degrees, $d$, which can occur with $e=7$. Following this, Sambale \cite{Sam} extended this classification of $d$ for all values when $e\leq 11$.  Much of the previous work has relied on computer algebra systems, with obvious limitations as $e$ gets larger.   \par
The results above concern determining when some fixed integer occurs as an irreducible character degree.  Motivated by these problems, we instead investigate the arithmetic structure forced by solvable groups admitting large abelian $\pi$-subgroups.\par  Previously, in \cite{Brandon}, we considered the case when $(d,d+e)=1$ and $d$ is square-free.  We proved that the existence of the character degree $d$ is equivalent to an explicit system of arithmetic congruences such that the product of the moduli is equal to our irreducible character degree, $D$.
In this paper, we extend these ideas in two directions.  First, we remove the square-free hypothesis on $d$, but we still assume that $d$ is an irreducible character degree of our group.  We prove the same arithmetic characterization continues to hold for arbitrary $d$-part. In particular, we show the following: 
\begin{theorem}\label{theorem 1}
    Let $\{d_1,\dots,d_m,p_1\dots,p_n\}$ be a set of distinct primes such that $k_j,a_i\in\mathbb{N}$ for all $i,j$.  Let $\pi=\{d_1,\dots,d_m\}$, and $D=d_1^{k_1}\cdots d_m^{k_m}$.
    Let $G$ be a solvable group such that $|G|=Dp_1^{a_1}\cdots p_n^{a_n}$, $D\in\text{cd}(G)$, and any $A\in\text{Hall}_\pi(G)$ is abelian.  Then there exist prime powers $h_1,\dots,h_t$ and $b_1,\dots,b_t\in\mathbb{N}$, such that $$ h_1\equiv 1(\text{mod }b_1),\dots, h_t\equiv 1(\text{mod }b_t),$$ where $b_1\cdots b_t=D$ and $h_1\cdots h_t\mid p_1^{a_1}\cdots p_n^{a_n}.$
\end{theorem}\par
Second, we consider the more general setting in which $d$ and $d+e$ need not be coprime.  In this setting, the arithmetic characterization no longer requires the assumption that $d$ is an irreducible character degree.  Instead, it follows from the existence of an abelian $\pi$-subgroup acting faithfully on the $\pi'$-part of the Fitting subgroup, where $\pi$ is the set of prime divisors of $d$. In particular, we show the following: 
\begin{theorem}\label{theorem 2}
    Let $\pi=\{d_1,\dots, d_m\}$ and $\omega=\{p_1,\dots ,p_n\}$ be sets of primes.  Let $D=d_1^{k_1}\cdots d_m^{k_m}$ where $k_j\in\mathbb{N}$ for all $j$.  Define $$\bar\omega=\omega\setminus\pi \text{   and   } \tau=\pi\cap\omega.$$  
    Suppose $G$ is a solvable group of order $Dp_1^{a_1}\cdots p_n^{a_n}$, where $a_i\in\mathbb{N}$ for all $i$, such that $A\leq G$ is an abelian $\pi$-subgroup of order $D$ such that $C_A(F_{\pi'})=1$ where $$F_{\pi'}=\prod_{q\not\in\pi}O_q(G)=\prod_{q\in\bar\omega}O_q(G).$$  Then there exist prime powers $h_1,\dots,h_t$ and $b_1,\dots,b_t\in\mathbb{N}$, such that $$ h_1\equiv 1(\text{mod }b_1),\dots, h_t\equiv 1(\text{mod }b_t),$$ where $b_1\cdots b_t=D$ and $h_1\cdots h_t\mid p_1^{a_1}\cdots p_n^{a_n}.$
\end{theorem}

\section{Background}
We begin by showing two lemmas that will be used in the proofs of the main results.  
\begin{lemma}
    Let $P$ be a finite $p$-group for some prime $p$ and $\alpha\in\text{Aut}(P)$ such that $(|\alpha|,p)=1$.  If $\alpha$ acts trivially on $P/\Phi(P)$, where $\Phi(P)$ is the Frattini subgroup of $P$, then $\alpha=1$.  Equivalently, if $H$ is a $p'$-group acting on $P$ and acting trivially on $P/\Phi(P)$, then $H$ acts trivially on $P$.  
\end{lemma}
\begin{proof}
    Let $\alpha\in\text{Aut}(P)$ have $p'$-order and suppose $\alpha$ acts trivially on $P/\Phi(P)$.  Let $\langle\alpha\rangle$ act on $P$.  Since $(|\alpha|,p)=1$, we have a coprime action and so $P=[P,\alpha]C_p(\alpha)$ where $[P,\alpha]=\langle x^{-1}x^\alpha:x\in P\rangle$.  Since $\alpha$ acts trivially on $P/\Phi(P)$, for every $x\in P$, we have $x^\alpha\phi(P)=x\Phi(P)$ and so $x^{-1}x^\alpha\in\Phi(P)$.  Thus $[P,\alpha]\leq \Phi(P)$ and so $P=[P,\alpha]C_P(\alpha)\leq\Phi(P)C_P(\alpha).$  Hence $P=\Phi(P)C_P(\alpha)$.  By the Frattini property \cite{FGT}, we have that $P=C_P(\alpha)$ and so $\alpha$ fixes every element of $P$.  Thus $\alpha=1$.
\end{proof}

Lemma 1 above allows us to obtain faithful coprime actions by passing to Frattini quotients. Now, we will show that an abelian group acting irreducibly on a finite vector space must be cyclic.
\begin{lemma}
    Let $V$ be a finite-dimensional vector space over $\mathbb{F}_q$, for some $q$, and let $G\leq \text{GL}(V)$ be an abelian subgroup acting irreducibly on $V$.  Then, $G$ is cyclic.
\end{lemma}
\begin{proof}
Let $V$ be an irreducible $\mathbb{F}_qG$-module, where $G\leq \text{GL}(V)$ is abelian.  Fix some $g\in G$.  Since $g$ acts linearly on $V$, we have a $\mathbb{F}_q$-linear map $v\mapsto gv$ for some $v\in V$.  Consider some arbitrary $h\in G$. Since $G$ is abelian, we have $$g(hv)=(gh)v=(hg)v=h(gv),$$ and so $g(hv)=h(gv)$.  Hence, for each $g\in G$, the linear transformation $v\mapsto gv$ is an element of $E=\text{End}_{\mathbb{F}_qG}(V)$.  By Schur's lemma \cite{CTFG}, since $V$ is irreducible, $E$ is a division ring.  Since $V$ is finite-dimensional over $\mathbb{F}_q$, $E$ is also finite.  Hence $E$ is a finite division ring and hence is a field by Wedderburn's Little theorem \cite{CTFG}.  Now, each $g\in G$ is in the multiplicative group of $E$, that is, $E^\times$, and so $G\leq E^\times$, hence $G$ is cyclic.  

\end{proof}

\section{Main Results}

We now extend the result from \cite{Brandon} by removing the assumption that $d$ is square-free. We restate Theorem 1 below for convenience: 
\begin{theoremone}
    Let $\{d_1,\dots,d_m,p_1\dots,p_n\}$ be a set of distinct primes such that $k_j,a_i\in\mathbb{N}$ for all $i,j$.  Let $\pi=\{d_1,\dots,d_m\}$, and $D=d_1^{k_1}\cdots d_m^{k_m}$.
    Let $G$ be a solvable group such that $|G|=Dp_1^{a_1}\cdots p_n^{a_n}$, $D\in\text{cd}(G)$, and any $A\in\text{Hall}_\pi(G)$ is abelian.  Then there exist prime powers $h_1,\dots,h_t$ and $b_1,\dots,b_t\in\mathbb{N}$, such that $$ h_1\equiv 1(\text{mod }b_1),\dots, h_t\equiv 1(\text{mod }b_t),$$ where $b_1\cdots b_t=D$ and $h_1\cdots h_t\mid p_1^{a_1}\cdots p_n^{a_n}.$
\end{theoremone}
     \begin{proof}
         Let $\chi\in\text{Irr}(G)$ such that $\chi(1)=D$.  Replace $G$ by $G/\ker\chi$, if necessary, to ensure that $\chi$ is faithful.  Since $G$ is solvable, Hall's theorem \cite{FGT} guarantees the existence of a Hall $\pi$-subgroup of $G$.  Let $A\in\text{Hall}_\pi(G)$.  By assumption, $A$ is abelian.  \par
         Suppose there exists $j$ such that $O_{d_j}(G)\neq 1$.  Then $O_{d_j}(G)$ lies in some $D_j\in\text{Syl}_{d_j}(G)$.  After conjugation, we may assume $D_j\leq A$ and so $O_{d_j}(G)\leq A$.  Since $A$ is abelian, $O_{d_j}(G)$ must be abelian.  Thus $O_{d_j}(G)$ is a nontrivial normal abelian subgroup of $G$.  By \cite{Mich}, we have $D=\chi(1)\mid |G:O_{d_j}(G)|$, a contradiction as \newline $|G:O_{d_j}(G)|_D<D$ .  Thus $O_{d_j}(G)=1$ for all $j$.  For all $i$, define $P_i=O_{p_i}(G)$.  Then $$F:=\textbf{F}(G)=P_1\times\cdots \times P_n,$$ where $\textbf{F}(G)$ is the Fitting subgroup of $G$.  \par 
         Since $G$ is solvable, $C_G(F)\leq F$.  If $x\in C_A(F)$, then $x\in C_G(F)\leq F$.  Since $A$ is a $\pi$-group, and $F$ is a $\pi'$-group, we have $x\in A\cap F=1$ thus $x=1$. Hence, $C_A(F)=1$ and $A$ acts faithfully on $F$.  \par 
         For each $i$, define $$V_i:=P_i/\Phi(P_i).$$  Since. for every $i$, $(|A|,p_i)=1$, Maschke's theorem \cite{CTFG} gives that the action of $A$ on each $V_i$ is completely reducible, and $$V_i=H_{i1}\oplus\cdots\oplus H_{ir_i},$$ where each $H_{ij}$ is an irreducible $\mathbb{F}_{p_i}A$-module.  Let $$\mathscr{H}=\{H_{ij}:1\leq i\leq n, 1\leq j\leq r_i\}.$$ \par 
         Suppose $a\in A$ acts trivially on every $V_i$.  That is, $a$ acts trivially on $P_i/\Phi(P_i)$ for every $i$.  Since $A$ is a $\pi$-group and each $P_i$ is a $\pi'$-group, we may apply Lemma 1 and so $a$ acts trivially on each $P_i$.  Thus $a$ acts trivially on $F$.  Since we have shown that $a\in C_A(F)=1$, we must have that $a=1$. Thus, $\bigcap_{i=1}^n C_A(V_i)=1$. Since $V_i=H_{i1}\oplus\cdots\oplus H_{ir_i}$, an element of $A$ centralizes every $H_{ij}$ if and only if that element centralizes every $V_i$.  Hence, $$\bigcap_{H\in\mathscr{H}}C_A(H)=1.$$ \par
         Set $A_0=A$.  If $A_0=1$, we have nothing to prove.  Otherwise, by above, we have $\bigcap_{H\in\mathscr{H}}C_{A_0}(H)=1$.  Thus, there exists some $H_1\in\mathscr{H}$ such that $C_{A_0}(H_1)\neq A_0$.  Define $$A_1:=C_{A_0}(H_1).$$  Then $A_1<A_0$ and we can continue in this way until we have constructed the chain $$A=A_0>A_1>\cdots >A_{r-1}.$$  If $A_{r-1}=1$, stop.  Otherwise, from the same reason as above, there exists some $H_r\in\mathscr{H}$ such that $C_{A_{r-1}}(H_r)\neq A_{r-1}$ and we define $$A_r:=C_{A_{r-1}}(H_r).$$  Since $A$ is finite, this process must eventually terminate, and so we get the chain $$A=A_0>A_1>\cdots >A_t=1,$$ where, for each $r$, we define $$B_r:=A_{r-1}/A_r.$$  Then, $\prod_{r=1}^t|B_r|=D$. For each $r$, let $b_r=|B_r|$.     \par 
          The chosen $H_1,\dots,H_t$ from above are all distinct.  Indeed, once $H_s$ has been chosen, we define $A_s:=C_{A_{s-1}}(H_s)$.  Every later subgroup $A_r$, with $r\geq s$, satisfies $A_r\leq A_s$.  Therefore, every later $A_r$ centralizes $H_s$ and so $H_s$ cannot be chosen again later because it would not give a proper centralizer.  Thus each $H_r$ is a distinct irreducible summand among the decompositions of all $V_i$.  Thus $h_1\cdots h_t\mid \prod_{i=1}^n|V_i|$, and since $|V_i|\mid |O_{p_i}(G)|$ for each $i$, we have that $h_1\cdots h_t\mid p_1^{a_1}\cdots p_n^{a_n}$. \par 
         By construction, $A_r=C_{A_{r-1}}(H_r)$ and so $A_r$ is the kernel of the action of $A_{r-1}$ on $H_r$.  Let $$\rho_r:A_{r-1}\rightarrow\text{GL}(H_r).$$ Then, $\ker\rho_r=A_r$, and so by the First Isomorphism theorem, $$B_r=A_{r-1}/A_r\cong \rho_r(A_{r-1}).$$  That is, $B_r$ acts faithfully on $H_r$.\par 
         Fix $r$ so that $H_r\leq V_i$ for some $i$.  Then, $H_r$ is an $\mathbb{F}_{p_i}$-vector space.  Let $\dim_{\mathbb{F}_{p_i}}H_r=e_r$, for some integer $e_r$.  If we consider the action of $A$ on $H_r$, we obtain the homomorphism $$\rho:A\rightarrow\text{GL}(H_r).$$  Since $A$ is abelian, $\rho(A)$ is an abelian subgroup of $\text{GL}(H_r)$.  We have also shown that $H_r$ is an irreducible $\mathbb{F}_{p_i}A$-module.  Since the action of $A$ on $H_r$ factors through $\rho(A)$, we have that $H_r$ is also an irreducible $\mathbb{F}_{p_i}\rho(A)$-module.  Hence, $\rho(A)$ is an abelian subgroup of $\text{GL}(H_r)$ acting irreducibly on $H_r$.  Using Lemma 2, we have that $\rho(A)$ is cyclic and so acts by multiplication via nonzero elements of a finite field.  \par
         Since $B_r\cong\rho_r(A_{r-1})\leq \rho(A)$, we have that $B_r$ is cyclic and act by multiplication by elements of $\mathbb{F}_{p_i^{e_r}}^\times$.  We will identify $H_r$ with the additive group of $\mathbb{F}_{p_i^{e_r}}$, and consider $B_r\leq \mathbb{F}_{p_i^{e_r}}^\times$ acting by multiplication. \par 
         Let $1\neq g\in B_r$.   Now, $g$ acts as multiplication by some scalar $1\neq \lambda\in \mathbb{F}_{p_i^{e_r}}^\times$.  Suppose $v\in H_r$ satisfies $g\cdot v=v$.  Then, $\lambda v=v$ and so $(\lambda-1)v=0$.  Since $\lambda\neq 1$, we must have that $v=0$. Therefore, the zero vector is the only fixed vector of $g$, hence $C_{H_r}(g)=0$, and so $B_r$ acts Frobeniusly on $H_r$.  Set $h_r:=|H_r|$.  We have $b_r\mid(h_r-1)$, and so $$h_r\equiv 1(\text{mod }b_r).$$ Since $r$ was arbitrary, we repeat this for each $B_r$ to obtain the desired sequence of congruences.
\end{proof}
In Theorem 1, the existence of the character degree $d$ is used to obtain a faithful action of the Hall $\pi$-subgroup on the relevant $\pi'$-part of the Fitting subgroup of $G$.  In Theorem 2, we will remove the condition that $(d,d+e)=1$.  We will then assume the existence of the aforementioned faithful action between an abelian $\pi$-subgroup on the relevant $\pi'$-part of the Fitting subgroup of $G$, but this allows us to forgo the assumption that $d$ is an irreducible character degree.  Note that our abelian subgroup in Theorem 2 is no longer Hall, because its order and index will not be coprime.  Since we are assuming faithfulness, the following essentially becomes a corollary of Theorem 1.
\begin{theoremtwo}
    Let $\pi=\{d_1,\dots, d_m\}$ and $\omega=\{p_1,\dots ,p_n\}$ be sets of primes.  Let $D=d_1^{k_1}\cdots d_m^{k_m}$ where $k_j\in\mathbb{N}$ for all $j$.  Define $$\bar\omega=\omega\setminus\pi \text{   and   } \tau=\pi\cap\omega.$$  
    Suppose $G$ is a solvable group of order $Dp_1^{a_1}\cdots p_n^{a_n}$, where $a_i\in\mathbb{N}$ for all $i$, such that $A\leq G$ is an abelian $\pi$-subgroup of order $D$ such that $C_A(F_{\pi'})=1$ where $$F_{\pi'}=\prod_{q\not\in\pi}O_q(G)=\prod_{q\in\bar\omega}O_q(G).$$  Then there exist prime powers $h_1,\dots,h_t$ and $b_1,\dots,b_t\in\mathbb{N}$, such that $$ h_1\equiv 1(\text{mod }b_1),\dots, h_t\equiv 1(\text{mod }b_t),$$ where $b_1\cdots b_t=D$ and $h_1\cdots h_t\mid p_1^{a_1}\cdots p_n^{a_n}.$
\end{theoremtwo}
\begin{proof}
    Note that in Theorem 1, the assumption that $D\in\text{cd}(G)$ implies $O_{d_j}(G)=1$ for each $j$.  Hence $F=\textbf{F}(G)=\prod_{i-1}^nO_{p_i}(G)$, and faithfulness of the action follows from the fact that $C_G(F)\leq F$.  Now, for Theorem 2, the primes in $\tau$ may divide both $|A|$ and $|\textbf{F}(G)|$.  Instead, let $F_{\pi'}=\prod_{q\in\bar\omega}O_q(G)$ and we assume $C_A(F_{\pi'})=1$, hence $A$ acts faithfully on $F_{\pi'}$.  For each prime $q\in\bar\omega$, define $P_q:=O_q(G)$.  Then, we will define $$V_q:=P_q/\Phi(P_q).$$  Since $(q,|A|)=1$, by Maschke's theorem \cite{CTFG}, we may write $$V_q=H_{q1}\oplus\cdots\oplus H_{qr_q},$$ where each $H_{qj}$ is an irreducible $\mathbb{F}_qA$-module.  Let $$\mathscr{H}=\{H_{qj}:q\in\bar\omega, 1\leq j\leq r_q\}.$$  Suppose $x\in\bigcap_{H\in\mathscr{H}}C_A(H)$.  Then, $x$ centralizes every irreducible summand $H_{qj}$, and so $x$ centralizes each $V_q$.  Thus, $x$ acts trivially on $P_q/\Phi(P_q)$ for every $q\in\bar\omega$.  Since $x\in A$, we know $(|x|,q)=1$ and so $x$ is a $q'$-automorphism of $P_q$.  By Lemma 2, $x$ must then act trivially on $P_q$ for each $q\in\bar\omega$. Thus, $x\in C_A(F_{\pi'})=1$.  Thus $x=1$, and so $\bigcap_{H\in\mathscr{H}}C_A(H)=1$.  From here, we construct the same centralizer chain as in Theorem 1 and repeat the same proof.
\end{proof}

\end{document}